\documentclass{amsart}
\usepackage{amsmath}
\usepackage{amsfonts}
\usepackage{amsthm}
\usepackage{amssymb}
\usepackage{cite}
\usepackage{enumerate}
\usepackage[all]{xy}
\usepackage{cite}
\usepackage{mathrsfs}
\usepackage{color}
\usepackage{xcolor}
\usepackage{hyperref}
\usepackage{fancyhdr}
\hypersetup{
	colorlinks=true,
	anchorcolor=blue,
	linkcolor=blue,
	filecolor=blue,
	urlcolor=blue,	
	citecolor=blue,
	bookmarks=true,
	bookmarksopen=true,
	pdfborder=000
}
\numberwithin{equation}{section}
\theoremstyle{plain}
\newtheorem{thm}{Theorem}[section]

\newtheorem{proposition}[thm]{Proposition}

\newtheorem{cor}[thm]{Corollary}

\newtheorem{lemma}[thm]{Lemma}
\newtheoremstyle{noparens}%
 {}{}%
 {\itshape}{}%
 {\bfseries}{.}%
 { }%
 {\thmname{#1}\thmnumber{ #2}\mdseries\thmnote{ #3}}
\theoremstyle{noparens}
\newtheorem{lemmaNoParens}[thm]{Lemma}
\newtheorem{thmNoParens}[thm]{Theorem}
\newtheorem{propositionNoParens}[thm]{Proposition}
\newtheorem{remarkNoParens}[thm]{Remark}
\theoremstyle{definition}
\newtheorem{defn}[thm]{Definition}
\theoremstyle{remark}
\newtheorem{rmk}[thm]{Remark}

\makeatletter
\newcommand{\rmnum}[1]{\romannumeral #1}
\newcommand{\Rmnum}[1]{\expandafter\@slowromancap\romannumeral #1@}
\newcommand{\li}{\lim\limits_{z\rightarrow\partial\Omega}}
\newcommand{\lin}{\lim\limits_{n\rightarrow\infty}}
\newcommand{\K}{K$\ddot{\operatorname{a}}$hler}
\makeatother
\begin{document}
\title{Asymptotic characterizations of Strong Pseudoconvexity on Pseudoconvex Domains of Finite Type in $\mathbb{C}^2$}
\author{Jinsong Liu\textsuperscript{1,2} $\&$ Xingsi Pu\textsuperscript{1,2} $\&$ Lang Wang\textsuperscript{1,2}}
\address{$1.$ HLM, Academy of Mathematics and Systems Science,
Chinese Academy of Sciences, Beijing, 100190, China}
\address{$2.$ School of
Mathematical Sciences, University of Chinese Academy of Sciences,
Beijing, 100049, China }
\subjclass[2020]{32T15, 32T25, 53C20}
\keywords{Strong pseudoconvexity, Holomorphic sectional curvature, Pseudoconvex domains, Finite type}
\email{liujsong@math.ac.cn,\:puxs@amss.ac.cn,\:wanglang2020@amss.ac.cn}

\begin{abstract}
In this paper, we provide some characterizations of strong pseudoconvexity by the boundary behavior of intrinsic invariants for smoothly bounded pseudoconvex domains of finite type in $\mathbb{C}^2$. As a consequence, if such domain is biholomorphically equivalent to a quotient of the unit ball, then it is strongly pseudoconvex.
\end{abstract}

\maketitle

\section{Introduction}
In several complex variables, a domain in $\mathbb{C}^n$ with $C^2$-smooth boundary is called strongly pseudoconvex if the Levi form of the boundary is positive definite. Strongly pseudoconvex domains form an important class of complex domains and many studies have focused on such domains in recent years. For instance, Fefferman \cite{fef} obtained the asymptotic formula of Bergman kernel near a boundary point of smoothly bounded strongly pseudoconvex domains. 

For bounded domains in $\mathbb{C}^n$, numerous results show that the asymptotic complex geometry of a bounded strongly pseudoconvex domain coincides with the unit ball. And main theorems in \cite{wong,squ,kim2,kle} describe the behavior of intrinsic invariants near boundary on such domains, which can be stated as follows.
\begin{thm}\label{int}
Suppose $\Omega\subset\mathbb{C}^n$ is a bounded strongly pseudoconvex domain with $C^2$-smooth boundary, then the following statements hold 

(1) $\li s_{\Omega}(z)=1$, where $s_{\Omega}(z)$ is the squeezing function of $\Omega$.

(2) $\li\frac{M_{\Omega}^C(z)}{M_{\Omega}^K(z)}=1$, where  $M_{\Omega}^K(z)$ and $M_{\Omega}^C(z)$ are respectively Kobayashi-Eisenman and Carath$\acute{e}$odory volume elements.

(3) $\li\frac{C_{\Omega}(z,v)}{K_{\Omega}(z,v)}=1$ and $\li\frac{B_{\Omega}(z,v)}{K_{\Omega}(z,v)}=\sqrt{n+1}$ uniformly for $0\neq v\in\mathbb{C}^n$, where $K_{\Omega},C_{\Omega},B_{\Omega}$ are respectively Kobayashi metric, Carath$\acute{e}$odory metric and Bergman metric.

(4) $\li H(B_{\Omega})=-\frac{4}{n+1}$, which means $\li\sup\limits_{0\neq X\in T_z\Omega}\left|H(B_{\Omega})(X)+\frac{4}{n+1}\right|=0$. Here $H(B_{\Omega})$ is the holomorphic sectional curvature of $B_{\Omega}$.

In particular, if $\Omega$ is biholomorphic to the unit ball $\mathbb{B}^n$, then all of intrinsic invariants mentioned above are constants in $\Omega$.

\end{thm}

On the other hand, whether the strong pseudoconvexity of a bounded domain can be characterized by its intrinsic geometry has been studied in recent years. In \cite{charac}, Zimmer firstly proved that for a bounded convex domain $\Omega\subset\mathbb{C}^n$ with $C^{2,\alpha}$-smooth boundary for some $\alpha>0$, if $\li s_{\Omega}(z)=1$, then $\Omega$ is strongly pseudoconvex. Moreover, Joo and Kim\cite{kim} proved the case of smoothly bounded pseudoconvex domains of finite type in $\mathbb{C}^2$. And Nikolov \cite{hex} extended this result to the case of h-extendible domains in $\mathbb{C}^n$.

Later Bracci, Gaussier and Zimmer \cite{pin} proved a version of the characterization related to holomorphic sectional curvature. They proved that for a bounded convex domain with $C^{2,\alpha}$-smooth boundary for some $\alpha>0$, if it admits a complete K$\ddot{\operatorname{a}}$hler metric $g$ such that $\li H(g)=-c$ with some constant $c>0$, then it is strongly pseudoconvex.

Expanding on their previous work, we consider the boundary behavior of intrinsic invariants for smoothly bounded pseudoconvex domains of finite type in $\mathbb{C}^2$, aiming to characterize strong pseudoconvexity on these domains. Our main theorem stated below serves as the converse to Theorem \ref{int} for such domains.
\begin{thm}\label{thm1}
Let $\Omega\subset\mathbb{C}^2$ be a smoothly  bounded pseudoconvex domain of finite type, the following statements are equivalent

(1) $\Omega$ is strongly pseudoconvex.

(2) $\li s_{\Omega}(z)=1$. Here $s_{\Omega}(z)$ is the squeezing function.

(3) $\li\frac{M_{\Omega}^C(z)}{M_{\Omega}^K(z)}=1$, where  $M_{\Omega}^K(z)$ and $M_{\Omega}^C(z)$ are respectively Kobayashi-Eisenman and Carath$\acute{e}$odory volume elements.

(4) $\li\frac{C_{\Omega}(z,v)}{K_{\Omega}(z,v)}=1$ and $\li\frac{B_{\Omega}(z,v)}{K_{\Omega}(z,v)}=\sqrt{3}$ uniformly for $0\neq v\in\mathbb{C}^2$, where $K_{\Omega},C_{\Omega},B_{\Omega}$ are respectively Kobayashi metric, Carath$\acute{e}$odory metric and Bergman metric.

(5) it admits a complete K$\ddot{a}$hler metric $g$ such that the following holds
\[
\lim_{z\rightarrow\partial\Omega}H(g)=-c
\]
for some constant $c>0$.

\end{thm}

Since the Bergman metric is invariant under biholomorphic mappings, then we can derive the following proposition directly from (5) in Theorem \ref{thm1} by considering the pullback metric of covering space.
\begin{proposition}\label{cor1}
Let $\Omega\subset\mathbb{C}^2$ be a smoothly bounded  pseudoconvex domain of finite type, if $\Omega$ is biholomorphically equivalent to a quotient of the unit ball, then it is strongly pseudoconvex.
\end{proposition}

Combining with \cite[Theorem A.2]{unif},  we obtain a similar result to the case of smoothly bounded pseudoconvex domains of finite type in $\mathbb{C}^2$ with real analytic boundary.
\begin{cor}
Let $\Omega\subset\mathbb{C}^2$ be a smoothly bounded  pseudoconvex domain of finite type with real analytic boundary, then it is biholomorphically equivalent to a quotient of the unit ball if and only if its boundary is spherical.
\end{cor}

This paper is orgnized as follows. In Sect.\ref{sec2} we give the preliminaries. In Sect.\ref{sec3} we recall the construction of a scaling sequence, and prove the stability of intrinsic invariants under the scaling sequence. Finally,  we prove Theorem \ref{thm1} in Sect.\ref{sec4}.

\section{Preliminaries}\label{sec2}
\subsection{Notations}(1) For $z\in\mathbb{C}^n$, let $|z|$ be the standard Euclidean norm of $z$.

(2) Suppose $z\in\mathbb{C}^n$ and $r>0$, for $n\geq2$ then $B_r(z):=\left\{w\in\mathbb{C}^n:|w-z|<r\right\}$  and  $\mathbb{B}^n$ is denoted as the unit ball of $\mathbb{C}^n$. For $n=1$, we let $\Delta$  be the unit disc of $\mathbb{C}$.

(3) Given a domain $\Omega\subsetneq\mathbb{C}^n$ and $z\in\Omega$, we let
\[
\delta_{\Omega}(z):=\inf\left\{|z-w|:w\in\partial\Omega\right\}.
\]

(4) For a $C^1$-smooth function $f:\Omega\subset\mathbb{C}^n\rightarrow\mathbb{C}$ and $0\neq X\in\mathbb{C}^n$, we denote
\[
X(f)(z):=\frac{d}{dt}\bigg|_{t=0}f\left(z+tX\right)
\]
for $z\in\Omega$.

\subsection{The Invariant Metrics}Recall that a $Finsler\ metric$ in $\Omega$ is an upper continuous map $F:\Omega\times\mathbb{C}^n\rightarrow[0,\infty)$ such that $F(z,tX)=|t|F(z,X)$ for all $z\in\Omega,\ t\in\mathbb{C},\ X\in\mathbb{C}^n$. And the corresponding distance $d_F$ is defined by
\begin{align*}
d_F(x,y):=\inf\{L_F(\gamma):\gamma:[0,1]\rightarrow\Omega\ \operatorname{is\ a\ piecewise}\ C^1\operatorname{-smooth\ curve}\\
\operatorname{with}\ \gamma(0)=x,\gamma(1)=y\},
\end{align*}
where $L_F(\gamma):=\int_0^1F(\gamma(t),\dot{\gamma}(t))dt.$

For a domain $\Omega\subset\mathbb{C}^n$, there are three important invariant metrics on it: Kobayashi metric, Carath$\acute{\operatorname{e}}$odory metric and Bergman metric. Letting $z\in\Omega,0\neq X\in\mathbb{C}^n$,  we define the Kobayashi infinitesimal metric  by
\[
K_{\Omega}(z,X):=\inf\{|\xi|:\exists{f\in \operatorname{Hol}(\Delta,\Omega),\ \operatorname{with}\ f(0)=z,d(f)_0(\xi)=X}\},
\]
 and the Carath$\acute{\operatorname{e}}$odory infinitesimal metric is defined by 
\[
C_{\Omega}(z,X):=\sup\{|(df)_z(X)|:f\in\operatorname{Hol}(\Omega,\Delta),\ f(z)=0 \}.
\]

Moreover, the Bergman kernel and Bergman metric of $\Omega$ are respectively defined as follows
\begin{align*}
&\kappa_{\Omega}(z,z):=\sup\{|f(z)|^2:f\in\operatorname{Hol}(\Omega,\mathbb{C}),||f||_{L^2(\Omega)}\leq1\},\\
&B_{\Omega}(z,X):=\frac{\sup\{|X(f)(z)|:f\in\operatorname{Hol}(\Omega,\mathbb{C}),f(z)=0,||f||_{L^2(\Omega)}\leq1\}}{\kappa^{1/2}_{\Omega}(z,z)}
\end{align*}
for $z\in\Omega$ and $X\in\mathbb{C}^n$.

For the sake of convenience in the rest of the paper, we let $d_{\Omega}(\cdot,\cdot)$ denote the Kobayashi distance on $\Omega$.

\begin{rmk}\label{rmk}
Suppose $\Omega\subset\mathbb{C}^2$ is a smoothly  bounded pseudoconvex domain of finite type, then Theorem 1 in \cite{cat} implies that there exist constants $A',A\geq1$ such that 

\begin{align*}\label{comp}
\dfrac{1}{A}C_{\Omega}(z,X)\leq &K_{\Omega}(z,X)\leq AC_{\Omega}(z,X),\\
\dfrac{1}{A'}B_{\Omega}(z,X)\leq &K_{\Omega}(z,X)\leq A'B_{\Omega}(z,X)
\end{align*}
for $z\in\Omega,\  X\in\mathbb{C}^2$.
\end{rmk}

\begin{defn}
 The Kobayashi-Eisenman volume element of a domain $\Omega\subset\mathbb{C}^n$ is defined by 
\[
M_{\Omega}^K(z):=\inf\left\{1/|\operatorname{det}f'(0)|^2:f\in\operatorname{Hol}(\mathbb{B}^n,\Omega),f(0)=z\right\},
\]
and the Carath$\acute{\operatorname{e}}$odory volume element is defined by
\[
M_{\Omega}^C(z):=\sup\left\{|\operatorname{det}f'(z)|^{2}:f\in\operatorname{Hol}(\Omega,\mathbb{B}^n),f(z)=0\right\}.
\]

Note that the $quotient\ invariant$ $q_{\Omega}(z):=\frac{M_{\Omega}^C(z)}{M_{\Omega}^K(z)}\leq1$ for all $z\in\Omega$ and it is invariant under biholomorphic mappings.
\end{defn}

\begin{rmk}\label{rmk1}
Let $\Omega$ be a complex manifold with complex dimension $n$, and $M_{\Omega}^K(z)=M_{\Omega}^C(z)\neq0$ for some $z\in\Omega$, then \cite[Theorem 1]{ball} implies that $\Omega$ is biholomorphic to the unit ball $\mathbb{B}^n$.
\end{rmk}

\subsection{Holomorphic Sectional Curvature} Suppose $(M,J,g)$ is a \K\ manifold, and let $R(g)$ be the curvature tensor of $(M,g)$. The $holomorphic\ sectional$ $curvature$ of non-zero tangent vector $X\in T_zM$ is given by
\[
H(g)(X):=\dfrac{R(X,JX,X,JX)}{g(X,X)^2}.
\]

$(M,g)$ is said to have $pinched\ negative\ holomorphic\ sectional\ curvature$, if there exist constants $a,b>0$ such that
\[
-a\leq H(g)(X)\leq -b
\]
for $z\in M$ and non-zero vector $X\in T_zM$. It is worth noting that if $(M,g)$ has pinched\ negative\ holomorphic\ sectional\ curvature and $g$ is complete, then \cite[Theorem 2]{yau} implies that $\sqrt{g(z)(X,X)}$ and the Kobayashi metric $K_M(z,X)$ are bi-Lipschitz equivalent.

From (4) of Theorem \ref{int}, we know that the Bergman metric on bounded strongly pseudoconvex domains has pinched negative holomorphic sectional curvature near boundary. Furthermore, Fu \cite{fu} showed that for smoothly bounded Reinhardt domains of finite type in $\mathbb{C}^2$, the Bergman metric also has pinched negative holomorphic sectional curvature near boundary.

\subsection{Squeezing Function} Following \cite{yeu,squ}, the $squeezing$\ $function$ of a bounded domain $\Omega\subset\mathbb{C}^n$ is defined by 
\begin{align*}
s_{\Omega}(z):=\sup\{r:&\operatorname{there\ exists\ a\ 1-1\ holomorphic\ mapping} f:\Omega\rightarrow\mathbb{B}^n\\
 &\operatorname{with} f(z)=0 \operatorname{and} B_r(0)\subset f(\Omega)\}.
\end{align*}

A domain $\Omega$ is said to be $holomorphically$\ $homogeneous\ regular$ or $uniformly$ $squeezing$, if there exists a constant $c>0$ such that $s_{\Omega}(z)\geq c$ for all $z\in\Omega$. From (1) in Theorem \ref{int}, we know that bounded strongly pseudoconvex domains with $C^2$-smooth boundary are holomorphically\ homogeneous\ regular. 

\begin{remarkNoParens}\label{est}
Let $\Omega\subset\mathbb{C}^n$ be a bounded domain. According to {\cite{squ,nik}}, the squeezing function and the invariants introduced earlier satisfy the following inequalities
\begin{align*}
s_{\Omega}(z)K_{\Omega}(z,v)\leq &C_{\Omega}(z,v)\leq K_{\Omega}(z,v),\\
s^{2n}_{\Omega}(z)M_{\Omega}^K(z)\leq &M_{\Omega}^C(z)\leq {s^{-2n}_{\Omega}(z)}M_{\Omega}^K(z),\\
2-2\frac{n+3}{n+1}s^{-4n}_{\Omega}(z)\leq H(&B_{\Omega})(z,X)\leq 2-2\frac{n+3}{n+1}s^{4n}_{\Omega}(z),\\
\sqrt{n+1}s^{n+1}_{\Omega}(z)K_{\Omega}(z,v)\leq &B_{\Omega}(z,v)\leq\sqrt{n+1}s^{-(n+1)}_{\Omega}(z)K_{\Omega}(z,v)
\end{align*}
for $z\in\Omega$ and $v,X\in\mathbb{C}^n$.
\end{remarkNoParens}

\section{Pseudoconvex domain of finite type in $\mathbb{C}^2$}\label{sec3}
In this section, we construct a scaling sequence with respect to a boundary point of $\Omega\subset\mathbb{C}^2$, and prove some properties of the invariants associated with the sequence. We always assume that this domain is of finite type at most $m$. Refering to \cite{gro,ver} for more details.

\subsection{Scaling Sequence}\label{sec3.1} Before constructing the scaling sequence, we firstly recall the concepts of $normal\ set$-$convergence$ and convergence of holomorphic mappings  introduced in \cite{kra}.  
\begin{defn}\label{converge}
For each $j=1,2,\cdots$, we suppose $\Omega_j$ is a domain in $\mathbb{C}^n$. The sequence $\left\{\Omega_j\right\}$ is said to $converge\ normally$ to a domain $\Omega'\subset\mathbb{C}^n$ if 

(1) For any compact set $K$, if there exists a constant $k$ such that $K$ is contained in the interior of $\bigcap\limits_{i>k}\Omega_i$, then $K\subset\Omega'$.

(2) For any compact set $K'\subset\Omega'$, there exists a constant $k>0$ such that $K'\subset\bigcap\limits_{i>k}\Omega_i$. 
\end{defn}

\begin{propositionNoParens}[{\cite[Proposition 9.2.3]{kra}}]\label{map}
Suppose $\Omega_j\subset\mathbb{C}^n$ converges normally to a domain $\Omega'$, then

(1) Suppose $\{f_j:\Omega_j\rightarrow\Omega_1\}$ is a sequence of holomorphic mappings converges uniformly on compact subsets of $\Omega'$, then its limit is a holomorphic mapping from $\Omega'$ to the closure of $\Omega_1$.

(2) Suppose $\{g_j:\Omega_2\rightarrow\Omega_j\}$ is a sequence of holomorphic mappings converges uniformly on compact subsets of $\Omega_2$, if $\Omega'$ is pseudoconvex and there exists a point $q\in\Omega_2$ such that $|det(dg_j)_q|>c$ holds for each $j$ with some constant $c>0$, then $g(z):=\lim\limits_{j\rightarrow\infty}g_j(z):\Omega_2\rightarrow\Omega'$ is a holomorphic mapping.  

\end{propositionNoParens}

Now we construct the scaling sequence with respect to a boundary point. Let $\Omega\subset\mathbb{C}^2$ be a smoothly bounded pseudoconvex domain of finite type and $p\in\partial\Omega$ be a boundary point. We may assume that $\Omega=\{(z,w)\in\mathbb{C}^2:\rho(z,w)<0\}$ with a smooth function $\rho(z,w)$ and $p=(0,0)$. Then \cite[Section 1]{hom} implies that after changing of coordinates, we can suppose that
\[
\rho(z,w)=\operatorname{Re}w+H_m(z)+o\left(|z|^{m+1}+|z||w|	\right)
\]
in a neighbourhood of $p$. Here $H_m(z):\mathbb{C}\rightarrow\mathbb{R}$ is a homogeneous polynomial of degree $m$, subharmonic and without harmonic terms. 

Suppose $\{p_n\}_{n\in\mathbb{N}}\subset\Omega$ converges to $p$, and for each $n$ we consider constant $\epsilon_n>0$ such that $\left(p_n^{(1)},p_n^{(2)}+\epsilon_n\right)\in\partial\Omega$ with $p_n=\left(p_n^{(1)},p_n^{(2)}\right)$. Let 
\[
\psi_n^{-1}(z,w):=\left(p_n^{(1)}+z,p_n^{(2)}+\epsilon_n+d_{n,0}w+\sum_{k=2}^md_{n,k}z^k\right)
\]
be  an automorphism of $\mathbb{C}^2$, where $d_{n,k}$ are chosen in the way such that 
\[
\rho\circ\psi_n^{-1}(z,w)=\operatorname{Re}w+P_n(z)+o\left(|z|^{m+1}+|z||w|\right).
\]
Here $P_n(z):\mathbb{C}\rightarrow\mathbb{R}$ is a subharmonic polynomial without harmonic terms and $P_n(0)=0$ with degree $m$. We select constant $\tau_n>0$ such that 
\[
||P_n(\tau_n\cdot)||=\epsilon_n,
\]
where $||\cdot||$ denote the norm in the space of polynomials with degree at most $m$. 

For each $n\geq1$, we define the following mapping 
\[
\delta_n(z,w):=\left(\frac{z}{\tau_n},\frac{w}{\epsilon_n}\right)
\]
and $\alpha_n:=\delta_n\circ\psi_n$, which is an automorphism of $\mathbb{C}^2$. Note that if we let $\Omega_n=\alpha_n(\Omega)$ for each $n=1,2,\cdots$, then $\Omega_n$ converges to $\widehat{\Omega}:=\left\{(z,w)\in\mathbb{C}^2:\operatorname{Re}w+P_{\infty}(z)<0\right\}$ in the sense of Definition \ref{converge}. Here $P_{\infty}(z):\mathbb{C}\rightarrow\mathbb{R}$ is a real-valued subharmonic polynomial without harmonic terms and its degree is $m$. Moreover, we have 
\[
q_n:=\alpha_n(p_n)=\left(0,-d_{n,0}^{-1}\right)\rightarrow(0,-1)=q_{\infty},
\]
and the domain $\widehat{\Omega}$ is called a $model\ domain$.

\subsection{Stability of Intrinsic Invariants}
Now we prove the stability of intrinsic invariants under the scaling sequence. We firstly recall the stability of Kobayashi metrics as follows.

\begin{lemmaNoParens}[{\cite[Lemma 5.2]{ver}}]\label{stability}
For $(z,v)\in\widehat{\Omega}\times\mathbb{C}^2$, we have 
\begin{equation}\label{stab}
\lim\limits_{n\rightarrow\infty}K_{\Omega_n}(z,v)=K_{\widehat{\Omega}}(z,v)
\end{equation}
and the convergence is uniform on compact subsets of $\widehat{\Omega}\times\mathbb{C}^2$.
\end{lemmaNoParens}

In the case of Carath$\acute{\operatorname{e}}$odory metrics, the family $\left\{C_{\Omega_n}\right\}$ is upper semicontinuous under the scaling sequence.

\begin{lemma}\label{cara}
For $(z,v)\in\widehat{\Omega}\times\mathbb{C}^2$, then
\[
\limsup\limits_{n\rightarrow\infty}C_{\Omega_n}(z,v)\leq C_{\widehat{\Omega}}(z,v).
\]
\end{lemma}

\begin{proof}
Fix $(z,v)\in\widehat{\Omega}\times\mathbb{C}^2$, we may assume that $z\in\Omega_n$ for each $n\geq1$. Suppose that $f_n(w):\Omega_n\rightarrow\Delta$ is an extremal mapping for $C_{\Omega_n}(z,v)$, then we obtain that $C_{\Omega_n}(z,v)=|(df_n)_z(v)|$ and $f_n(z)=0$. Montel's theorem implies that after passing to a subsequence, $f(w):=\lim\limits_{n\rightarrow\infty}f_n(w):\widehat{\Omega}\rightarrow\Delta$ is holomorphic with $f(z)=0$ and $(df)_z(v)=\lim\limits_{n\rightarrow\infty}(df_n)_z(v)$. And it implies that 
\[
\limsup\limits_{n\rightarrow\infty}C_{\Omega_n}(z,v)\leq C_{\widehat{\Omega}}(z,v),
\]
which completes the proof.
\end{proof}

Inspired by \cite[Lemma 5.7]{ver}, we establish the following result concerning the stability of Kobayashi distances under the scaling sequence.

\begin{lemma}\label{kob}
Let $\Omega_n,\widehat{\Omega}$ be as above, then
\[
\lim\limits_{n\rightarrow\infty}d_{\Omega_n}(z,w)=d_{\widehat{\Omega}}(z,w)
\]
holds for $(z,w)\in\widehat{\Omega}\times\widehat{\Omega}$. And the convergence is uniform on compact sets of $\widehat{\Omega}\times\widehat{\Omega}$.
\end{lemma}

\begin{proof}
In the proof of this lemma, we denote $B_{\Omega}(z,r)$ the Kobayashi ball in $\Omega$ with radius $r$ centered at $z$. For a contradiction, we may assume that there exists a compact set $K\subset\widehat{\Omega}\times\widehat{\Omega}$ and $(z_j,w_j)\in K$ such that 
\[
\left|d_{\Omega_j}(z_j,w_j)-d_{\widehat{\Omega}}(z_j,w_j)\right|>\epsilon_0
\]
with some $\epsilon_0>0$ for large $j$. We also assume that $z_j\rightarrow z_0,w_j\rightarrow w_0$ and all $z_j,w_j,z_0,w_0$ lie in $ K'$ for some compact set $K'$ of $\widehat{\Omega}$. Note that for sufficiently large $j$, we have
\begin{equation}\label{contra}
\left|d_{\Omega_j}(z_j,w_j)-d_{\widehat{\Omega}}(z_0,w_0)\right|>\epsilon_0/2.
\end{equation}

Similar to the proof of inequality (5.15) in \cite[Lemma 5.7]{ver}, we obtain that 
\[
d_{\Omega_j}(z_j,w_j)<d_{\widehat{\Omega}}(z_0,w_0)+\epsilon
\]
for fixed $\epsilon>0$ sufficiently small and $j$ large.

Now we prove that for each $r>0$ then
\[
B_{\widehat{\Omega}}(z_0,R-r)\subset B_{\Omega_j}(z_j,R)
\]
with all $R>0$ and all $j$ large. To see this, it is enough to show that for each $s\in\widehat{\Omega}$
\[
\limsup_{j\rightarrow\infty}d_{\Omega_j}(z_j,s)\leq d_{\widehat{\Omega}}(z_0,s).
\]
Now we fix $s\in\widehat{\Omega}$ and choose a piecewise $C^1$-smooth curve $\gamma:[0,1]\rightarrow\widehat{\Omega}$, connecting $z_0$ and $s$, such that $L_{K_{\widehat{\Omega}}}(\gamma)<d_{\widehat{\Omega}}(z_0,s)+r/2$. Let $\gamma_j(t):[0,1]\rightarrow\Omega_j$ as 
\[
\gamma_j(t):=\gamma(t)+(z_j-z_0)(1-t),
\]
then $\gamma_j(0)=z_j$ and $\gamma_j(1)=s$. Moreover, we know that $\gamma_j\rightarrow\gamma$ and $\gamma_j'\rightarrow\gamma'$ uniformly on $[0,1]$. Then Lemma \ref{stability} implies that 
\[
d_{\Omega_j}(z_j,s)\leq\int_0^1K_{\Omega_j}\left(\gamma_j(t),\gamma_j'(t)\right)dt\leq\int_0^1K_{\widehat{\Omega}}\left(\gamma(t),\gamma'(t)\right)dt+r/2<d_{\widehat{\Omega}}(z_0,s)+r
\]
for all $j$ large and $r$ small. 

The completeness of $K_{\widehat{\Omega}}$ on $\widehat{\Omega}$ implies that there exist constants $r>0,r'>0$ such that 
\[
K'\subset B_{\widehat{\Omega}}(z_0,r)\subset B_{\Omega_j}(z_j,r')
\]
for all $j$ large. Then \cite[Lemma 5.7]{ver} deduces that for all $w\in B_{\Omega_j}(z_j,r')$ 
\[
d_{\Omega_j}(w,q_j)\leq d_{\Omega_j}(z_j,w)+d_{\Omega_j}(z_j,q_j)\leq d_{\Omega_j}(z_j,w)+M
\]
with some constant $M>0$, independently of $w$ and $j$. Hence $B_{\Omega_j}(z_j,r')\subset B_{\Omega_j}(q_j,r'+M)$ and $B_{\Omega_j}(z_j,r')$ is compactly contained in $\widehat{\Omega}$ for large $j$ by \cite[Lemma 5.6]{ver}.

 Following a similar proof of inequality (5.16) in \cite[Lemma 5.7]{ver}, we obtain 
\[
d_{\widehat{\Omega}}(z_0,w_0)\leq d_{\Omega_j}(z_j,w_j)+C\epsilon
\]
for $j$ large with some constant $C>0$. This means that $\lim\limits_{j\rightarrow\infty}d_{\Omega_j}(z_j,w_j)=d_{\widehat{\Omega}}(z_0,w_0)$ and contradicts inequality (\ref{contra}).
\end{proof}

\begin{lemma}\label{quo}
For $z\in\widehat{\Omega}$, then 
\[
\lin M_{\Omega_n}^K(z)=M_{\widehat{\Omega}}^K(z).
\]
And the convergence is uniform on compact sets of $\widehat{\Omega}$.
\end{lemma}

\begin{proof}
For a contradiction, we may assume that there exists a compact set $K\subset\widehat{\Omega}$ and $\{z_j\}\subset K$ such that 
\[
\left|M_{\Omega_j}^K(z_j)-M_{\widehat{\Omega}}^K(z_j)\right|>\epsilon_0
\]
for some $\epsilon_0>0$. Moreover, we may assmue that $z_j\rightarrow z_0\in K$ and $K\subset\Omega_j$ for each $j$.  Since $M_{\widehat{\Omega}}^K(z)$ is continuous on $\widehat{\Omega}$, then we obtain that
\[
\left|M_{\Omega_j}^K(z_j)-M_{\widehat{\Omega}}^K(z_0)\right|>\epsilon_0/2
\]
for all $j$ large.

Note that for each $j$, there exists a holomorphic mapping $g_j(z):\mathbb{B}^2\rightarrow\Omega_j$ such that 
\[
g_j(0)=z_j,\left|\operatorname{det}g_j'(0)\right|^{-2}=M_{\Omega_j}^K(z_j).
\]
For fixed $r<1$, we have
\[
\sup_{w\in g_j(B_r(0))}d_{\Omega_j}(z_j,w)\leq \sup_{\zeta\in B_r(0)}d_{\mathbb{B}^2}(0,\zeta)=d_{\mathbb{B}^2}(0,r)
\]
and for each $w\in g_j(B_r(0))$
\[
d_{\Omega_j}(q_j,w)\leq d_{\Omega_j}(z_j,w)+d_{\Omega_j}(q_j,z_j)\leq M
\]
with some constant $M>0$ depending only on $r$. Hence \cite[Lemma 5.6]{ver} implies that $g_j(B_r(0))$ is compactly contained in $\widehat{\Omega}$ for large $j$.

From Lemma \ref{kob}, we know that $ g_j(B_r(0))$ is uniformly bounded, then Proposition \ref{map} and Montel's theorem imply that $g(z):=\lim\limits_{j\rightarrow\infty
	} g_j(z):\mathbb{B}^2\rightarrow\widehat{\Omega}$ is holomorphic. Since $g(0)=z_0$ and $\left|\operatorname{det}g'(0)\right|^{-2}=\lim\limits_{j\rightarrow\infty}\left|\operatorname{det}g_j'(0)\right|^{-2}$, then we obtain that $M_{\widehat{\Omega}}^K(z_0)\leq\liminf\limits_{j\rightarrow\infty} M_{\Omega_j}^K(z_j)$.

On the other hand, suppose $g(z):\mathbb{B}^2\rightarrow\widehat{\Omega}$ is an extremal mapping for $M_{\widehat{\Omega}}^K(z_0)$. For fixed $0<r<1$, we consider the mapping
\[
g_j(z)=g((1-r)z)+z_j-z_0,
\]
then $g_j(0)=z_j$ and $g_j(\mathbb{B}^2)\subset\Omega_j$ for $j$ large. Note that 
\[
\left|\operatorname{det}g_j'(0)\right|=(1-r)^2\left|\operatorname{det}g'(0)\right|,
\]
hence $\limsup\limits_{j\rightarrow\infty}M_{\Omega_j}^K(z_j)\leq M_{\widehat{\Omega}}^K(z_0)$ by letting $r\rightarrow0$, and $\lim\limits_{j\rightarrow\infty}M_{\Omega_j}^K(z_j)=M_{\widehat{\Omega}}^K(z_0)$. This is a contradiction.
\end{proof}

\begin{rmk}\label{rmk2}
By using a similar argument with Lemma \ref{cara}, we know that for each $z\in\widehat{\Omega}$ then
\[
\limsup\limits_{n\rightarrow\infty}M_{\Omega_n}^C(z)\leq M_{\widehat{\Omega}}^C(z).
\]

\end{rmk}

\subsection{Quasi-bounded Geometry}To prove Theorem \ref{thm1}, we need the notion of $quasi$-$bounded\ geometry$ introduced by Yau and Cheng in \cite{cheng}. Further details can be found in \cite{yau}. 

\begin{defn}\label{qb}
		Let $(M,g)$ be a K$\ddot{\operatorname{a}}$hler manifold of complex dimension $n$. We say that $(M,g)$ has $quasi$-$bounded\ geometry$ if there exist constants $r_2>r_1>0$ such that: for any point $z\in M$ there exsits a neighbourhood $U\subset\mathbb{C}^n$ and a nonsingular holomorphic mapping  $\varphi:U\rightarrow M$ such that

(1) $\varphi(0)=z$.

(2) $B_{r_1}(0)\subset U\subset B_{r_2}(0)$.

(3) There exists a constant $C\geq1$ determined only by $r_1,r_2,n$ such that 
\[
\dfrac{1}{C}g_0\leq\varphi^{*}g\leq Cg_0,
\]
 where $g_0$ is the standard Euclidean metric on $\mathbb{C}^n$.

(4) For any integer $q\geq0$, there exists a constant $A_q$ determined by $q,r_1,r_2,n$ such that
\begin{equation}\label{qbg}
\sup_{x\in U}\left|\dfrac{\partial^{|\mu|+|\nu|}(\varphi^{*}g)_{i\bar{k}}}{\partial z^{\mu}\bar{\partial}z^{\nu}}(x)\right|\leq A_q\ \operatorname{for\ all} |\mu|+|\nu|\leq q \operatorname{and} 1\leq i,k\leq n,
\end{equation}
where $(\varphi^{*}g)_{i\bar{k}}$ is the component of $\varphi^{*}g$ on $U$ in terms of the natural coordinates $z=(z_1,\cdots,z_n)$ and $\mu,\nu$ are the multiple indices with $|\mu|=\mu_1+\cdots\mu_n$. Moreover, the mapping $\varphi$ is called a $quasi$-$coordinate\ map$ and $(U,\varphi)$ is called a $quasi$-$coordinate\ chart$ of $M$.
\end{defn}

The following theorem gives a characterization that whether a complex manifold $(M,g)$ has quasi-bounded geometry.

\begin{thmNoParens}[{\cite[Theorem 9]{yau}}]\label{bound}
Let $(M,g)$ be a complete K$\ddot{a}$hler manifold with complex dimension $n$. Then $(M,g)$ has quasi-bounded geometry if and only if for each integer $q\geq0$, there exists a constant $C_q>0$ such that 
\[
\sup_{M}||\nabla^qR(g)||_g\leq C_q.
\]
Here $R(g)$ is the curvature tensor of $g$ and the constants $r_1,r_2$ and $C$ in Definition \ref{qb} can be selected in the way that they are determined only by $C_0$.
\end{thmNoParens}

Similar to \cite[Proposition 6.1]{pin},  by using the notion of quasi-bounded geometry presented above, we are able to prove the following proposition regarding the convergence of K$\ddot{\operatorname{a}}$hler metrics.

\begin{proposition}\label{kah}
Let $\{\Omega_n\}$ and $\widehat{\Omega}$ be the scaling sequence and model domain in Section \ref{sec3.1}, and suppose that $g_n$ is a complete K$\ddot{a}$hler metric on $\Omega_n$ for each $n$. If

(1) There exists a constant $A\geq1$, independently of $n$, such that
\[
\dfrac{1}{A}K_{\Omega_n}(z,v)\leq\sqrt{g_n(z)(v,v)}\leq AK_{\Omega_n}(z,v)
\]
for all $z\in\Omega_n$ and $v\in\mathbb{C}^2$.

(2) For any integer $q\geq0$, there exists a constant $C_q>0$, independently of $n$, such that 
\[
\sup_{\Omega_n}||\nabla^q R(g_n)||_{g_n}\leq C_q.
\]
Then after passing to a subsequence $g_n$ converges locally uniformly in $C^{\infty}$ topology to a complete K$\ddot{a}$hler metric $g_{\infty}$ on $\widehat{\Omega}$.
\end{proposition}
\begin{proof}
To prove this proposition, we may need to show that: if $K\subset\widehat{\Omega}$ is any compact set and $\mu,\nu$ are multi-indices, then there exist constants $J\geq0,C(K,\mu,\nu)\geq1$ determined only by $K,\mu,\nu$, such that 
\[
\sup_{n\geq J}\max_{x\in K}\left|\dfrac{\partial^{|\mu|+|\nu|}(g_n)_{i\bar{k}}}{\partial z^{\mu}\bar{\partial}z^{\nu}}(x)\right|\leq C(K,\mu,\nu).
\]

For a contradiction, we may suppose that there exists a compact set $K\subset\widehat{\Omega}$ and multi-indices $\mu,\nu$ such that: for any integer $l\geq0$, there exists a constant $n_l\geq l$ and $x_l\in K$ such that 
\[
\left|\dfrac{\partial^{|\mu|+|\nu|}(g_{n_l})_{i\bar{k}}}{\partial z^{\mu}\bar{\partial}z^{\nu}}(x_l)\right|\geq l.
\]

We may assume that $x_l\rightarrow x_{\infty}\in\widehat{\Omega}$ and $K\subset\Omega_{n_l}$ for all $l\geq1$ because of Definition \ref{converge}. From Theorem \ref{bound}, we obtain that each $(\Omega_n,g_n)$ has quasi-bounded geometry with constants independently of $n$. Hence for any integer $l$, there is a domain $U_l\subset\mathbb{C}^2$ and a nonsingular mapping $\varphi_l:U_l\rightarrow\Omega_{n_l}$ with $\varphi_l(0)=x_l$ satisfying conditions in Definition \ref{qb} with uniform constants $r_1,r_2,C$ and $\{A_q\}_{q\geq0}$.

For fixed $r<r_1$, then 
\begin{align*}
\sup_{w\in\varphi_l(B_r(0))}d_{\Omega_{n_l}}(x_l,w)\leq\sup_{\zeta\in B_r(0)}d_{B_{r_1}(0)}(0,\zeta)= d_{\mathbb{B}^2}\left(0,\frac{r}{r_1}\right).
\end{align*}
Note that for each $w\in\varphi_l(B_r(0))$, we have that 
\[
d_{\Omega_{n_l}}(q_{n_l},w)\leq d_{\Omega_{n_l}}(q_{n_l},z_l)+d_{\Omega_{n_l}}(z_l,w)\leq d_{\Omega_{n_l}}(q_{n_l},z_l)+d_{\mathbb{B}^2}\left(0,\frac{r}{r_1}\right)\leq M
\]
with some constant $M>0$, independently of $l$. Hence from \cite[Lemma 5.6]{ver} each $\varphi_l(B_r(0))$ is compactly contained in $\widehat{\Omega}$  for all $l$ large.

 Combining with Lemma \ref{kob}, we know that each image of $\varphi_l$ on $B_r(0)$ is uniformly bounded. And from condition (3) in Definition \ref{qb}, we obtain that 
\[
\dfrac{1}{C}|v|\leq\sqrt{g_{n_l}(x_l)((d\varphi_l)_0(v),(d\varphi_l)_0(v))}
\]
for each integer $l$ and $v\in\mathbb{C}^2$. It implies that 
\[
\dfrac{1}{C}|v|\leq\sqrt{g_{n_l}(x_l)((d\varphi_l)_0(v),(d\varphi_l)_0(v))}\leq AK_{\Omega_{n_l}}(x_l,(d\varphi_l)_0(v))\leq A\dfrac{|(d\varphi_l)_0(v)|}{\delta_{\Omega_{n_l}}(x_l)}.
\]
and
\[
|(d\varphi_l)_0(v)|\geq \dfrac{\delta_{\Omega_{n_l}}(x_l)}{AC}|v|\rightarrow\dfrac{\delta_{\widehat{\Omega}}(x_{\infty})}{AC}|v|.
\]
Then Proposition \ref{map} and Montel's theorem imply that $\varphi:=\lim\limits_{l\rightarrow\infty}\varphi_l:B_{r_1}(0)\rightarrow\widehat{\Omega}$ is holomorphic and nonsingular at 0.

Note that we can select a neighbourhood $U_1$ such that $\varphi|_{U_1}$ is invertible. Since $\varphi_l$ converges to $\varphi$ in $C^{\infty}$ topology, then there exists a constant $N>0$ such that $\varphi_l|_{U_2}$ is invertible whenever $l\geq N$. Here $U_2$ is a neighbourhood of 0 such that $\bar{U}_2\subset U_1$. Moreover, we know that $\left(\varphi_l|_{U_2}\right)^{-1}$ converges locally uniformly to $\left(\varphi|_{U_2}\right)^{-1}$.

Now we choose a neighbourhood $V$ of $x_{\infty}\in\widehat{\Omega}$ such that $V\subset\varphi(U_2)$. Without loss of generality, we may assume that $V\subset\varphi_l\left(U_2\right)$ for all $l\geq N$ and 
\[
\sup_{l\geq N}\sup_{x\in V}\left|\dfrac{\partial^{|a|+|b|}(\varphi_l|_{U_2})^{-1}}{\partial z^a\bar{\partial}z^b}(x)\right|<\infty
\]
whenever $|a|+|b|\leq|\mu|+|\nu|$. After increasing $N$ if necessarily, we may assmue that $x_l\in V$ whenever $l\geq N$. Combining with  condition (4) in Definition \ref{qb}, we obtain that 
\[
\sup_{l\geq N}\left|\dfrac{\partial^{|\mu|+|\nu|}(g_{n_l})_{i\bar{k}}}{\partial z^{\mu}\bar{\partial}z^{\nu}}(x_l)\right|<\infty,
\]
which is a contracdiction.

After passing to a subsequence, we may assume that $g_n$ converges locally uniformly to a Hermitian metric $g_{\infty}$ on $\widehat{\Omega}$ in $C^{\infty}$ topology. Then Lemma \ref{stab} implies that
\[
\sqrt{g_{\infty}(z)(v,v)}=\lim\limits_{n\rightarrow\infty}\sqrt{g_n(z)(v,v)}\geq\dfrac{1}{A}\lim\limits_{n\rightarrow\infty}K_{\Omega_n}(z,v)=\dfrac{1}{A}K_{\widehat{\Omega}}(z,v)>0
\]
and
\[
\sqrt{g_{\infty}(z)(v,v)}=\lim\limits_{n\rightarrow\infty}\sqrt{g_n(z)(v,v)}\leq{A}\lim\limits_{n\rightarrow\infty}K_{\Omega_n}(z,v)={A}K_{\widehat{\Omega}}(z,v),
\]
hence $g_{\infty}$ is a complete K$\ddot{\operatorname{a}}$hler metric on $\widehat{\Omega}$. And it completes the proof.
\end{proof}

By using the previous proposition, we are able to deduce the following result concerning the stability of Bergman metrics under the scaling sequence.

\begin{proposition}\label{ber}
Let $B_n$ be the Bergman metric on $\Omega_n$, then after passing to a subsequence, there exists a complete K$\ddot{a}$hler metric $B_{\infty}$ such that $B_n$ converges to $B_{\infty}$ on $\widehat{\Omega}$ locally uniformly in $C^{\infty}$ topology.
\end{proposition}
\begin{proof}
Suppose $\Omega\subset\mathbb{C}^2$ is a smoothly bounded pseudoconvex domain of finite type and $B_{\Omega}$ is the Bergman metric. Then the main theorem in \cite{mcn} implies that $H(B_{\Omega})$ is bounded from above and below, and hence the same as sectional curvature of $B_{\Omega}$. Combining with estimates of Bergman metric in \cite{cat} and \cite[Proposition 2.1]{moduli}, we obtain that $\left(\Omega,B_{\Omega}\right)$ has positive injectivity radius. 

Note that from inequalities (5) and (10) in \cite{mcn}, then there exists a constant $c>0$ such that $\kappa_{\Omega}(z,z)/B^2_{\Omega}(z)\geq c$ for all $z\in\Omega$. Here $B^2_{\Omega}$ is the volume form on $\Omega$ with respect to Bergman metric. Following a similar proof of \cite[Lemma 26]{yau}, then $\left(\Omega,B_{\Omega}\right)$ has quasi-bounded geometry.

 Since Bergman metric is invariant under biholomorphic mappings, then each $(\Omega_n,B_{n})$ has quasi-bounded geometry with uniform constants in Theorem \ref{bound}. It means that for each $q\geq0$ there exists a constant $C_q>0$, independently of $n$, such that
\[
\sup_{\Omega_n}||\nabla^q R(B_n)||_{B_n}\leq C_q
\]
holds for each $n$. Moreover, Remark \ref{rmk} implies that there exists a constant $A\geq1$, independently of $n$, such that 
\[
\frac{1}{A}K_{\Omega_n}(z,v)\leq B_n(z,v)\leq AK_{\Omega_n}(z,v)
\]
for each $z\in\Omega_n$ and $v\in\mathbb{C}^2$. Hence by using Proposition \ref{kah}, we can complete the proof.
\end{proof}

\section{Proof of Theorem \ref{thm1}}\label{sec4}Now we are ready to prove our main theorem. The following theorem plays a crucial role in our proof.
\begin{thmNoParens}[{\cite[Theorem 4]{berg}}]\label{berg}
Let $M$ be a simply connected complex manifold with complex dimension $n$, and for any $\epsilon>0$ there exists a complete K$\ddot{a}$hler metric g on $M$ such that 
\[
-1-\epsilon\leq H(g)\leq -1+\epsilon.
\]
Then $M$ is biholomorphic to the unit ball in $\mathbb{C}^n$.
\end{thmNoParens}

 Inspired by \cite{uni,Huang}, we deduce the following theorem, which can imply the strong pseudoconvexity of the model domain $\widehat{\Omega}$ if it is biholomorphic to a quotient of the unit ball.
\begin{thm}\label{sph}
Suppose $\Omega\subset\mathbb{C}^2$ is a pseudoconvex domain of finite type with real-analytic boundary, and assume that it admits a plurisubharmonic defining function $\phi:\bar{\Omega}\rightarrow(-\infty,0]$. If it is covered by $\mathbb{B}^2$, then $\partial\Omega$ is spherical. 
\end{thm}

\begin{proof}
Suppose $\pi:\mathbb{B}^2\rightarrow\Omega$ is a covering map, $p\in\partial\Omega$ is an arbitrary boundary point. Let $V$ be a coordinate neighbourhood of $p$ such that $V\cap\Omega$ is simply connected, then the germ of $\pi^{-1}$ extends to a biholomorphic mapping $f$ from $V\cap\Omega$ to an open subset of $\mathbb{B}^2$.

 By applying Hopf lemma to the plurisubharmonic function $\phi\circ\pi$ on $\mathbb{B}^2$, then there exist constants $c_1,c_2>0$, independently of $x\in V\cap\Omega$, such that 
\[
c_1\delta_{\mathbb{B}^2}(f(x))\leq |\phi\circ\pi(f(x))|=|\phi(x)|\leq c_2\delta_{\Omega}(x)
\]
for all $x\in V\cap\Omega$.

 We know that the map $f$ extends to $\partial\Omega\cap V$ as a H$\ddot{\operatorname{o}}$lder continuous map sending $\partial\Omega\cap V$ to the unit sphere $\partial\mathbb{B}^2$, by using argument based on the asymptotic behaviour of Kobayashi metric of $\mathbb{B}^2$(see also \cite[section 2.1]{uni} or \cite[section 5.1]{pinch}). Note that the extension of $f$ is non-constant by uniqueness theorem of boundary and both $\partial\Omega$ and $\partial\mathbb{B}^2$ are real analytic. Then Theorem A in \cite{Huang} implies that $f$ can be extended as a proper holomorphic mapping, still denoted by $f$, in a neighbourhood of $p$ that preserves the sides of $\partial\Omega$ and $\partial\mathbb{B}^2$. It means that there exists a neighbourhood $V_1\subset\mathbb{C}^2$ of $p$ such that $f$ extends a holomorphic mapping denoted by $f:V_1\rightarrow f(V_1)$, which is proper and preserves the sides of $\partial\Omega$ and $\partial\mathbb{B}^2$. Moreover, we have that $f(V_1\cap\Omega)\subset\mathbb{B}^2$ and $f(V_1\setminus\bar{\Omega})\subset f(V_1)\setminus\bar{\mathbb{B}}^2$. Since $f:V_1\rightarrow f(V_1)$ is proper, then \cite[Theorem 15.1.9]{rudin} implies that there exists an integer $k\geq1$ such that
\begin{align*}
&\# f^{-1}(w)=k,\ w\in f(V_1)\setminus S,\\
&\# f^{-1}(w)<k,\ w\in S.
\end{align*}
Here $\#f^{-1}(w)$ means the cardinality of $f^{-1}(w)$ and $S$ is the set of critical values of $f|_{V_1}$. Note that $f|_{V_1\cap\Omega}:V_1\cap\Omega\rightarrow f(V_1\cap\Omega)$ is biholomorphic, hence $k=1$ and $f$ is biholomorphic from $V_1$ to its image. This implies that $\partial\Omega$ is spherical.
\end{proof}

The following proposition implies that if $(\Omega,g)$ satisfies condition (5) of Theorem \ref{thm1}, then $\sqrt{g(z)(v,v)}$ and $K_{\Omega}(z,v)$ are bi-Lipschitz equivalent. 

\begin{proposition}\label{compa}
Let $\Omega\subset\mathbb{C}^2$ be a smoothly bounded pseudoconvex domain of finite type and $g$ be a complete K${\ddot{a}}$hler metric on $\Omega$. Suppose $(\Omega,g)$ has pinched negative holomorphic sectional curvature outside a compact subset, then there exists a constant $A\geq1$ such that 
\[
\dfrac{1}{A}K_{\Omega}(z,v)\leq \sqrt{g(z)(v,v)}\leq AK_{\Omega}(z,v)
\]
for all $z\in\Omega$ and $v\in\mathbb{C}^2$.
\end{proposition}

\begin{proof}
Note that there exists a compact set $K\subset\Omega$ and positive constants $a,b>0$ such that 
\[
-a\leq H(g)\leq -b
\]
holds outside $K$. Then from Lemma 19 in \cite{yau}, there exists a constant $A_1>0$ such that 
\[
\sqrt{g(z)(v,v)}\leq A_1K_{\Omega\setminus K}(z,v)
\]
for all $z\in\Omega\setminus K$ and $v\in\mathbb{C}^2$.

From inequality (4.1) in \cite[Section 4]{geh}, we know that there exists a compact subset $K'\subset\Omega$ and a constant $A_2\geq1$ such that $K\subset K'$ and
\[
K_{\Omega\setminus K}(z,v)\leq A_2K_{\Omega}(z,v)
\]
for $z\in\Omega\setminus K'$ and $v\in\mathbb{C}^2$. Hence we obtain that $\sqrt{g(z)(v,v)}\leq A_1A_2K_{\Omega}(z,v)$ for $z\in\Omega\setminus K',v\in\mathbb{C}^2$. Since $K'$ is compact, there exists a constant $A_3>0$ such that $\sqrt{g(z)(v,v)}\leq A_3K_{\Omega}(z,v)$ for $z\in K',v\in\mathbb{C}^2$. Hence we obtain
\[
\sqrt{g(z)(v,v)}\leq \max\{A_1A_2,A_3\}K_{\Omega}(z,v)
\]
for $z\in\Omega,v\in\mathbb{C}^2$.

On the other hand, for $z\in\Omega$ and $v\in\mathbb{C}^2$, let $f_1:\Omega\rightarrow\Delta$ be the extremal mapping of $C_{\Omega}(z,v)$. Note that the Ricci curvature of $g$ is bounded from below. Considering Yau-Schwarz Lemma \cite[Theorem 2]{ys} with respect to $f_1$, we know that there exists a constant $A_4>0$ such that 
\[
\sqrt{g(z)(v,v)}\geq A_4C_{\Omega}(z,v).
\]
Combining with Remark \ref{rmk}, there exists a constant $A_5>0$ such that 
\[
\sqrt{g(z)(v,v)}\geq A_5K_{\Omega}(z,v),
\]
which completes the proof.
\end{proof}

\noindent$Proof\ of\ Theorem\ \ref{thm1}.$ To prove the main theorem, it is enough to prove $(3)\Longrightarrow(1),(4)\Longrightarrow(1)$ and $(5)\Longrightarrow(1)$. For any $p\in\partial\Omega$, we assmue that $p=(0,0)$ and $\{\Omega_n\}$ is the scaling sequence, constructed in Section \ref{sec3.1}, with respect to $p$ for $\Omega_n=\alpha_n(\Omega)$.
\medskip

\noindent $(3)\Longrightarrow(1):$ We now consider the quotient invariant $q_{\widehat{\Omega}}(z)$ on $\widehat{\Omega}$. For fixed $w\in\widehat{\Omega}$, we have that $q_{\widehat{\Omega}}(w)\leq1$. From Lemma \ref{quo} and Remark \ref{rmk2}, we obtain
\[
\limsup\limits_{n\rightarrow\infty}q_{\Omega_n}(w)\leq q_{\widehat{\Omega}}(w)\leq 1.
\]
Since $q_{\Omega_n}(w)$ is biholomorphically invariant and $\li q_{\Omega}(z)=1$, then 
\begin{align*}\label{limit}
q_{\Omega_n}(w)=q_{\Omega}(\alpha^{-1}_n(w))\rightarrow 1.
\end{align*}
It means that $q_{\widehat{\Omega}}(w)=1$ and Remark \ref{rmk1} implies that $\widehat{\Omega}$ is biholomorhpic to $\mathbb{B}^2$. Note that $\widehat{\Omega}$ has real-analytic boundary and $r(z,w)=\operatorname{Re}w+P_{\infty}(z)$ is plurisubharmonic since $P_{\infty}(z)$ is subharmonic, hence $\partial\widehat{\Omega}$ is spherical by Theorem \ref{sph}. Moreover, we obtain that $P_{\infty}(z)=c_1|z|^2$ for some constant $c_1>0$. By using the same argument in \cite[Section 4]{kim}, we obtain that the domain $\Omega$ is strongly pseudoconvex at $p$. And it completes the proof.
\medskip

\noindent$(4)\Longrightarrow(1)$: For $(z,v)\in\widehat{\Omega}\times\mathbb{C}^2$, we have $K_{\widehat{\Omega}}(z,v)=C_{\widehat{\Omega}}(z,v)$. To see this, we may assume $0\neq v$ and observe that 
\[
\limsup\limits_{n\rightarrow\infty}\dfrac{C_{\Omega_n}(z,v)}{K_{\Omega_n}(z,v)}\leq \dfrac{C_{\widehat{\Omega}}(z,v)}{K_{\widehat{\Omega}}(z,v)}\leq1
\]
from Lemma \ref{stability} and Lemma \ref{cara}. Note that $\li\frac{C_{\Omega}(z,v)}{K_{\Omega}(z,v)}=1$ uniformly for $0\neq v$ implies that for any $\epsilon>0$, there exists a constant $\delta>0$ such that
\[
 1-\epsilon<\frac{C_{\Omega}(z,v)}{K_{\Omega}(z,v)}< 1+\epsilon
\]
for any $0\neq v$ and $z\in\Omega$ with $\delta_{\Omega}(z)<\delta$. Since $\alpha_n^{-1}(z)$ converges to $p\in\partial\Omega$, hence for large $n$ we obtain $\delta_{\Omega}(\alpha_n^{-1}(z))<\delta$. Then 
\[
\dfrac{C_{\Omega_n}(z,v)}{K_{\Omega_n}(z,v)}=\dfrac{C_{\Omega}(\alpha_n^{-1}(z),(d\alpha_n^{-1})v)}{K_{\Omega}(\alpha_n^{-1}(z),(d\alpha_n^{-1})v)}>1-\epsilon
\]
for large $n$. Taking limit for $n$ implies that $C_{\widehat{\Omega}}(z,v)=K_{\widehat{\Omega}}(z,v)$. Moreover, in this case, we also have $\lim\limits_{n\rightarrow\infty}C_{\Omega_n}(z,v)=C_{\widehat{\Omega}}(z,v)$ for $z\in\widehat{\Omega}$ and $v\in\mathbb{C}^2$. To see this, we may assume that there exists $(z_0,v_0)\in\widehat{\Omega}\times\mathbb{C}^2$ with $v_0\neq0$ such that 
\[
1>\lim_{n\rightarrow\infty}\frac{C_{\Omega_n}(z_0,v_0)}{C_{\widehat{\Omega}}(z_0,v_0)}
\]
after passing to a subsequence. However, we have
\[
\lim_{n\rightarrow\infty}\frac{C_{\Omega_n}(z_0,v_0)}{C_{\widehat{\Omega}}(z_0,v_0)}=\lim_{n\rightarrow\infty}\frac{C_{\Omega_n}(z_0,v_0)}{K_{\widehat{\Omega}}(z_0,v_0)}=\lim_{n\rightarrow\infty}\frac{C_{\Omega_n}(z_0,v_0)}{K_{\Omega_n}(z_0,v_0)}=1.
\]
This is a contradiction.

 From Proposition \ref{ber}, the Bergman metrics $B_n$ on $\Omega_n$ converges to a complete K$\ddot{\operatorname{a}}$hler metric $B_{\infty}$ locally uniformly on $\widehat{\Omega}$. More precisely, $\li\frac{B_{\Omega}(z,v)}{K_{\Omega}(z,v)}=\sqrt{3}$ uniformly for $0\neq v$ implies that $\sqrt{B_{\infty}(z)(v,v)}=\sqrt{3}K_{\widehat{\Omega}}(z,v)$ for all $z\in\widehat{\Omega}$ and $v\in\mathbb{C}^2$ by using the same argument as above. Since $K_{\widehat{\Omega}}$ is a K$\ddot{\operatorname{a}}$hler metric, the main result in \cite{stanton} implies that $\widehat{\Omega}$ is biholomorphic to the unit ball. Therefore, Theorem \ref{sph} implies that it is strongly pseudoconvex, and $\Omega$ is also strongly pseudoconvex at $p$. This completes the proof.
\medskip

\noindent $(5)\Longrightarrow(1)$: Suppose $(\Omega,g)$ satisfies condition (5) of Theorem \ref{thm1} and we may assume that $c=1$. Then there exists a compact set $K$ such that 
\[
-\dfrac{3}{2}\leq H(g)(X)\leq -\dfrac{1}{2}
\] 
for all $z\in\Omega\setminus K$ and $0\neq X\in T_z\Omega$. Note that there exists a constant $a>0$ such that
\[
||R(g)(z)||_g\leq a
\]
for all $z\in\Omega$, Indeed, equality (6.1) in \cite{gold} implies that $||R(g)(z)||_g\leq a_1$ for all $z\in\Omega\setminus K$ with some constant $a_1>0$. Since $K$ is compact, there exists a constant $a_2>0$ such that $||R(g)(z)||_g\leq a_2$ for $z\in K$.

 From Theorem 7.1 in \cite{pin}, we know that for any $\epsilon>0$ and $r>0$, there exists a complete K$\ddot{\operatorname{a}}$hler metric $h$ on $\Omega$ such that 

($\rmnum{1}$) $g$ and $h$ are $(1+\epsilon)$-bi-Lipschitz.

($\rmnum{2}$) for each $q\geq0$, there exists a constant $C_q>0$ such that
\[
\sup_{z\in\Omega}||\nabla^qR(h)||_h\leq C_q.
\]

($\rmnum{3}$) if $z\in\Omega$, then
\begin{equation}\label{sec}
\inf_{B_g(z,r)}H(g)-\epsilon\leq H(h)\leq\sup_{B_{g}(z,r)}H(g)+\epsilon.
\end{equation}
Here $B_g(z,r)$ is the metric ball with respect to $g$ with radius $r$ centered at $z$.  Then Proposition \ref{compa} implies that there exists a constant $A'\geq1$ such that
\[
\dfrac{1}{A'}K_{\Omega}(z,v)\leq\sqrt{h(z)(v,v)}\leq A'K_{\Omega}(z)(v,v)
\] 
for all $z\in\Omega$ and $v\in\mathbb{C}^2$.

Now we consider the pullback metric $h_n:=(\alpha_n^{-1})^*h$ on $\Omega_n$, which is a complete K$\ddot{\operatorname{a}}$hler metric. Since Kobayashi metric is invariant under biholomorphic mapping, then we obtain
\[
\dfrac{1}{A'}K_{\Omega_n}(z,v)\leq\sqrt{h_n(z)(v,v)}\leq A'K_{\Omega_n}(z,v)
\]
for all $z\in\Omega_n$ and $v\in\mathbb{C}^2$. Moreover, for each $n\geq 1$ and $q\geq0$, we have 
\[
\sup_{\Omega_n}||\nabla^qR(h_n)||_{h_n}=\sup_{\Omega}||\nabla^qR(h)||_h\leq C_q	. 
\]

From Proposition \ref{kah}, after passing to a subsequence we may assume that $h_n$ converges locally uniformly to a complete K$\ddot{\operatorname{a}}$hler metric $h_{\infty}$ on $\widehat{\Omega}$. Inequality (\ref{sec}) implies that $-1-\epsilon\leq H(h_{\infty})\leq -1+\epsilon$ on $\widehat\Omega$. 

Let $\Sigma$ be the universal covering space and $\pi:\Sigma\rightarrow\widehat{\Omega}$ be the covering map, we now prove it is $\mathbb{B}^2$. Considering the pullback metric $\pi^*h_{\infty}$ on $\Sigma$, then it is a complete K$\ddot{\operatorname{a}}$hler metric and 
\[
-1-\epsilon\leq H(\pi^*h_{\infty})\leq -1+\epsilon
\]
on $\Sigma$. Since $\epsilon$ is arbitray, then Theorem \ref{berg} implies that $\Sigma$ is biholomorphic to $\mathbb{B}^2$, and $\widehat{\Omega}$ is covered by $\mathbb{B}^2$. Applying Theorem \ref{sph}, we conclude that $\widehat{\Omega}$ is strongly pseudoconvex, which means that $\Omega$ is also strongly pseudoconvex at $p$. This completes the proof.
$\hfill\qed$

\vspace*{5mm}
\noindent {\bf Funding}.
The research of the first author was supported by National Key R\&D Program of China (Grant No. 2021YFA1003100), NSFC (Grant No. 11925107, 12226334), Key Re-search Program of Frontier Sciences, CAS (Grant No. ZDBS-LY-7002).

\bibliography{ref}

\begin{thebibliography}{10}

\bibitem{hom}
E.~Bedford and S.~Pinchuk.
\newblock Domains in $\mathbb{C}^2$ with noncompact automorphism groups.
\newblock {\em Indiana Univ. Math. J.}, 47(1):199--222, 1998.

\bibitem{gold}
R.~L. Bishop and S.~I. Goldberg.
\newblock Some implications of the generalized {G}auss-{B}onnet theorem.
\newblock {\em Trans. Amer. Math. Soc.}, 112:508--535, 1964.

\bibitem{pin}
F.~Bracci, H.~Gaussier, and A.~Zimmer.
\newblock The geometry of domains with negatively pinched
  {K}$\ddot{\operatorname{a}}$hler metrics.
\newblock {\em \url{arXiv:1810.11389}}, 2018.

\bibitem{cat}
D.~W. Catlin.
\newblock Estimates of invariant metrics on pseudoconvex domains of dimension
  two.
\newblock {\em Math. Z.}, 200(3):429--466, 1989.

\bibitem{cheng}
S.~Y. Cheng and S.~T. Yau.
\newblock On the existence of a complete {K}$\ddot{\operatorname{a}}$hler
  metric on non-compact complex manifolds and the regularity of {F}efferman’s
  equation.
\newblock {\em Comm. Pure Appl. Math.}, 33:507--544, 1980.

\bibitem{squ}
F.~Deng, Q.~Guan, and L.~Zhang.
\newblock Properties of squeezing functions and global transformations of
  bounded domains.
\newblock {\em Trans. Amer. Math. Soc.}, 368:2679--2696, 2016.

\bibitem{fef}
C.~Fefferman.
\newblock The {B}ergman kernel and biholomorphic mappings of pseudoconvex
  domains.
\newblock {\em Invent. Math.}, 26:1--65, 1974.

\bibitem{gro}
M.~Fiacchi.
\newblock Gromov hyperbolicity of pseudoconvex finite type domains in
  $\mathbb{C}^2$.
\newblock {\em Math. Ann.}, 382(1-2):37--68, 2022.

\bibitem{fu}
S.~Fu.
\newblock Geometry of {R}einhardt domains of finite type in $\mathbb{C}^2$.
\newblock {\em J. Geom. Anal.}, 6(3):407--431, 1996.

\bibitem{ball}
I.~Graham and H.~Wu.
\newblock Characterizations of the unit ball $\mathbb{B}^n$ in complex
  {E}uclidean space.
\newblock {\em Math. Z.}, 189(4):449--456, 1985.

\bibitem{kra}
R.~E. Greene, K.~T. Kim, and S.~G. Krantz.
\newblock The geometry of complex domains.
\newblock In {\em Progress in Math.}, volume 291.
  Birkh$\ddot{\operatorname{a}}$user, 2010.

\bibitem{berg}
R.~E. Greene and S.~G. Krantz.
\newblock Stability properties of the {B}ergman kernel and curvature properties
  of bounded domains.
\newblock In {\em Recent developments in several complex variables ({P}roc.
  {C}onf., {P}rinceton {U}niv., {P}rinceton, {N}. {J}., 1979)}, volume No. 100
  of {\em Ann. of Math. Stud.}, pages 179--198. Princeton Univ. Press,
  Princeton, NJ, 1981.

\bibitem{Huang}
X.~Huang.
\newblock Schwarz reflection principle in complex spaces of dimension two.
\newblock {\em Commun. Partial Differ. Equ.}, 21:1781--1828, 1996.

\bibitem{kim}
S.~Joo and K.~T. Kim.
\newblock On boundary points at which the squeezing function tends to one.
\newblock {\em J. Geom. Anal.}, 28:2456--2465, 2018.

\bibitem{kim2}
K.~T. Kim and J.~Yu.
\newblock Boundary behavior of the {B}ergman curvature in strictly pseudoconvex
  polyhedral domains.
\newblock {\em Pacific J. Math.}, 176:141--163, 1996.

\bibitem{kle}
P.~Klembeck.
\newblock K$\ddot{\operatorname{a}}$hler metrics of negative curvature, the
  {B}ergmann metric near the boundary, and the {K}obayashi metric on smooth
  bounded strictly pseudoconvex sets.
\newblock {\em Indiana Univ. Math. J.}, 27:275--282, 1978.

\bibitem{geh}
H.~Li, X.~Pu, and H.~Wang.
\newblock The {G}ehring-{H}ayman type theorem on pseudoconvex domains of finite
  type in $\mathbb{C}^2$.
\newblock {\em \url{arXiv:2310.10306 }}, 2023.

\bibitem{moduli}
K.~Liu, X.~Sun, and S.~T. Yau.
\newblock Canonical metrics on the moduli space of {R}iemann surfaces {II}.
\newblock {\em J. Differ. Geom.}, 69(1):163--216, 2005.

\bibitem{ver}
P.~Mahajan and K.~Verma.
\newblock Some aspects of the {K}obayashi and {C}arath\'{e}odory metrics on
  pseudoconvex domains.
\newblock {\em J. Geom. Anal.}, 22(2):491--560, 2012.

\bibitem{mcn}
J.~D. McNeal.
\newblock Holomorphic sectional curvature of some pseudoconvex domains.
\newblock {\em Proc. Amer. Math. Soc.}, 107:113--117, 1989.

\bibitem{unif}
S.~Nemirovski and R.~Shafikov.
\newblock Uniformization of strictly pseudoconvex domains. {I}.
\newblock {\em Izv. Math.}, 69:1189--1202, 2005.

\bibitem{uni}
S.~Nemirovski and R.~Shafikov.
\newblock Uniformization of strictly pseudoconvex domains. {II}.
\newblock {\em Izv. Math.}, 69:1203--1210, 2005.

\bibitem{hex}
N.~Nikolov.
\newblock Behavior of the squeezing function near h-extendible boundary points.
\newblock {\em Proc. Amer. Math. Soc.}, 146(8):3455--3457, 2018.

\bibitem{nik}
N.~Nikolov and M.~Trybuła.
\newblock Estimates for the squeezing function near strictly pseudoconvex
  boundary points with applications.
\newblock {\em J. Geom. Anal.}, 30:1359--1365, 2020.

\bibitem{pinch}
S.~Pinchuk.
\newblock Holomorphic {M}aps in $\mathbb{C}^n$ and the {P}roblem of
  {H}olomorphic {E}quivalence.
\newblock In {\em Several complex variables. III. Geometric function theory (G.
  M. Khenkin and R. V. Gamkrelidze, eds.)}, pages 173--199. Encycl. Math. Sci.,
  vol. 9, Springer, Berlin, 1989.

\bibitem{rudin}
W.~Rudin.
\newblock Function theory in the unit ball of $\mathbb{C}^n$.
\newblock Springer, New York, 1980.

\bibitem{stanton}
C.~M. Stanton.
\newblock A characterization of the ball by its intrinsic metrics.
\newblock {\em Math. Ann.}, 264(2):271--275, 1983.

\bibitem{wong}
B.~Wong.
\newblock Characterization of the unit ball in $\mathbb{C}^n$ by its
  automorphism group.
\newblock {\em Invent. Math.}, 41(3):253--257, 1977.

\bibitem{yau}
D.~Wu and S.~T. Yau.
\newblock Invariant metrics on negatively pinched complete
  {K}$\ddot{\operatorname{a}}$hler manifolds.
\newblock {\em J. Amer. Math. Soc.}, 33:103--133, 2020.

\bibitem{ys}
S.~T. Yau.
\newblock A general {S}chwarz lemma for {K}$\ddot{\operatorname{a}}$hler
  manifolds.
\newblock {\em Amer. J. Math.}, 100:197--203, 1978.

\bibitem{yeu}
S.~K. Yeung.
\newblock Geometry of domains with the uniform squeezing property.
\newblock {\em Adv. Math.}, 221(2):547--569, 2009.

\bibitem{charac}
A.~Zimmer.
\newblock Characterizing strong pseudoconvexity, obstructions to
  biholomorphisms, and {L}yapunov exponents.
\newblock {\em Math. Ann.}, 374(3-4):1811--1844, 2019.

\end{thebibliography}
\bibliographystyle{plain}{}
\end{document}